\documentclass[reqno]{amsart}

\usepackage{amsmath, amssymb, amsthm, graphicx}

\newtheorem{theorem}{Theorem}[section]

\theoremstyle{definition}

\theoremstyle{remark}
\newtheorem{remark}{Remark}[section]

\numberwithin{equation}{section}

\newcommand{\p}{\partial}
\newcommand{\fei}{\varphi}

\makeatletter
\newcommand{\rmnum}[1]{\mathrm{\romannumeral #1}}
\newcommand{\Rmnum}[1]{\mathrm{\expandafter\@slowromancap\romannumeral#1@}}
\makeatother
\begin{document}

\title[Uniqueness and Instability of Subsonic--Sonic Flow]{Uniqueness and Instability of
Subsonic--Sonic Potential Flow in A Convergent Approximate Nozzle}

\author{Pan Liu and  Hairong Yuan}

\address{P. Liu:
Department of Mathematics, East China Normal University, Shanghai
200241, China} \email{pliu@math.ecnu.edu.cn}
\address{H. Yuan (Corresponding author):
Department of Mathematics, East China Normal University, Shanghai
200241, China} \email{hryuan@math.ecnu.edu.cn;\ \
hairongyuan0110@gmail.com}



\date{\today}

\begin{abstract}
We proved uniqueness and instability of the symmetric
subsonic--sonic flow solution of the compressible potential flow
equation in a surface with convergent areas of cross--sections. Such
a surface may be regarded as an approximation of a two--dimensional
convergent nozzle in aerodynamics. Mathematically these are
uniqueness and nonexistence results of a nonlinear degenerate
elliptic equation with Bernoulli type boundary conditions. The proof
depends on maximum principles and a generalized Hopf boundary point
lemma which was proved in the paper.
\end{abstract}

\maketitle

%%%----------------------------------------------------------------------------

\section{Introduction}

This paper is a continuation of our study on transonic flows in
nozzles \cite{Yu1,YH}. A fundamental difficulty to this problem is
the lack of  special solutions in a physical nozzle. However, as
motivated by \cite{SS}, it is rather easy to construct various
physically interesting special solutions, such as those
(supersonic--subsonic) transonic shocks and (subsonic--supersonic)
transonic flows observed in physical nozzles \cite{CoF}, for the
steady compressible Euler equations in certain Riemannian manifolds
\cite{Yu1,YH}. Therefore these manifolds may be regarded as
approximate nozzles and the study of the special flows in them would
be of certain importance for the understanding of flows in physical
nozzles. In the previous works \cite{CY,LY1,LY}, we and our
collaborators have studied stability and uniqueness of transonic
shocks for the potential flow equation or the complete Euler system,
which involve free boundary problems of uniformly elliptic equations
(for potential flow equation) or elliptic--hyperbolic composite
system (for Euler system). In \cite{Yu1,YH} we construct special
transonic flow solution and show its stability for potential flow
equation, which is a small perturbation result on nonlinear
equations of mixed type. The present paper is devoted to the
uniqueness and instability of subsonic--sonic flow in a convergent
approximate nozzle in the class of all $C^2$ functions satisfying
certain reasonable restrictions. These are global results on
degenerate elliptic equations. We first formulate the problem.\\

Let $\mathbf{S}^1$ be the standard unit circle in $\mathbf{R}^2$,
and $\mathcal{M}$ the Riemannian manifold $\{(x,y)\in
(0,1)\times\mathbf{S}^1\}$ with a metric $G=dx\otimes
dx+n(x)^2dy\otimes dy.$ Here $n(t)$ is a positive smooth function on
$[0,1]$ satisfies:
\begin{center}
$(H_1)$\qquad\qquad  $(\rmnum{1})$ $n''(t)>0;$\ \  $(\rmnum{2})$
$n'(t)<0$ for $t\in(0,1)$, $n'(1)=0.$
\end{center}
Without loss of generality, we also assume that $n(1)=1.$ We call
$\mathcal{M}$ a convergent approximate nozzle. In the following, we
take $(x,y)\in M=(0,1)\times(0,2\pi)$ as a coordinate system in
$\mathcal{M}$. Note that on $y=0$ and $y=2\pi$ there should pose
periodic conditions.

The potential flow equation governing steady irrotational isentropic
perfect gas flows in $\mathcal{M}$ is (c.f. \cite{YH})
\begin{eqnarray}\label{0101}
&&n(x)^2(c^2-(\p_1\varphi)^2)\p_{11}\varphi-2\p_1\varphi\p_2\varphi\p_{12}\varphi
+\left(c^2-\frac{1}{n(x)^2}(\p_2\varphi)^2\right)\p_{22}\varphi\nonumber\\
&&\qquad\qquad
+n(x)n'(x)\left(c^2+\frac{1}{n(x)^2}(\p_2\varphi)^2\right)\p_1\varphi=0.
\end{eqnarray}
(Here and hereafter we set $\p_1=\p_x,\ \p_2=\p_y$.) Direct
computation yields that this equation is of elliptic type if the
flow is subsonic ($c^2>(\p_1\varphi)^2+(\p_2\varphi)^2/n(x)^2),$ and
degenerate if the flow is sonic
($c^2=(\p_1\varphi)^2+(\p_2\varphi)^2/n(x)^2).$ We may compute the
speed of sound $c$  by the Bernoulli's law
\begin{eqnarray}\label{0102}
\frac{c^2}{\gamma-1}+\frac{(\p_1\varphi)^2+(\p_2\varphi)^2/n(x)^2}{2}
=\frac{c_0}{\gamma-1},
\end{eqnarray}
with $\gamma>1$ and $c_0>0$ two given constants. Let $p,\ \rho$ be
the pressure and density of the flow respectively. We use
$p=\rho^\gamma$ as the state function and then
$c=\sqrt{\gamma\rho^{\gamma-1}}.$

We call $\Sigma_{k}=\{k\}\times \mathbf{S}^1,\ k=0,1,$ respectively
the entry and exit of the approximate nozzle $\mathcal{M}.$
Obviously $\p\mathcal{M}=\Sigma_{0}\cup\Sigma_1$. In the following
we assume that
\begin{center}
$(H_2)$\qquad\qquad  $\p_1\fei\ge0$\ \ on both $\Sigma_0$ and
$\Sigma_1$,
\end{center}
which means exactly that the gas flows in $\mathcal{M}$ on
$\Sigma_0$ and flows out $\mathcal{M}$ on $\Sigma_1.$

We study subsonic--sonic flows in $\mathcal{M}$. The boundary
conditions are
\begin{eqnarray}
&(\p_1\varphi)^2+\frac{(\p_2\varphi)^2}{n(0)^2}=b_0^2 & \text{on}
\ \ \Sigma_0,\label{0103}\\
&(\p_1\varphi)^2+(\p_2\varphi)^2=b_1^2:=\frac{2c_0}{\gamma+1} &
\text{on}\ \ \Sigma_1, \label{0104}
\end{eqnarray}
where $b_0,\, b_1$ are positive constants. Note that \eqref{0103}
and \eqref{0104} are Bernoulli type conditions and the choice of
$b_1$ shows that the flow is sonic at the exit. We choose $b_0$ to
satisfy the algebraic equation
\begin{eqnarray}
b_0^{\gamma-1}(c_0-\frac{\gamma-1}{2}b_0^2)n(0)^{\gamma-1}=b_1^{\gamma+1}
\end{eqnarray}
and $b_0<b_1$ so that the flow is subsonic at the entry. It is
straightforward to show that there exists such a unique $b_0$. We
pose these nonlinear Bernoulli conditions rather than Dirichlet or
Neumann conditions, since from physical point of view, prescribing
density (pressure) at the entry and exit are more reasonable
\cite{CY,CoF}. By Bernoulli's law \eqref{0102}, this consideration
leads to \eqref{0103} and \eqref{0104}.\\

Now suppose the flow is symmetric, that is, depends only on $x$.
Then \eqref{0101}--\eqref{0104} are reduced to this problem of
$u(x)=\varphi'(x)$:
\begin{eqnarray}
&\p_x(n\rho u)=0, & x\in [0,1],\\
&{\displaystyle \frac{\gamma
\rho^{\gamma-1}}{\gamma-1}+\frac{u^2}{2}=\frac{c_0}{\gamma-1}}, & x\in [0,1],\\
& u=b_0, & x=0,\\
& u=b_1, & x=1.
\end{eqnarray}
Under the above choice of $b_0$, the velocity $u$ can be easily
solved and the obtained flow is subsonic in $\mathcal{M}\cup
\Sigma_0$, sonic on $\Sigma_1$, and always accelerates in
$\mathcal{M}$ (i.e., $u'(x)=\fei''(x)>0$). Therefore we may choose a
potential function $\fei_b=\fei_b(x)$ with $\fei_b'=u$  being a
special solution to problem \eqref{0101}--\eqref{0104}.

Apart from stability, it arises also naturally the question that
whether this special solution is unique in the large. That is, under
suitable assumptions, whether any solution to
\eqref{0101}--\eqref{0104}, modulo a constant, must be the $\fei_b$
obtained above. The following results provide a positive answer to
the uniqueness, but a negative one to the stability.

\begin{theorem}\label{thm1}
$(\rmnum{1})$\ The only $C^2(\bar{\mathcal{M}})$ solution to problem
\eqref{0101}--\eqref{0104} which satisfies $(H_2)$  and
\begin{eqnarray}\label{acc}
&\p_{11}\fei>0 & \text{on}\ \Sigma_1
\end{eqnarray}
is the symmetric solution $\fei_b$ modulo a constant.

$(\rmnum{2})$\ Under the same assumptions as in $(\rmnum{1})$, but
replacing \eqref{0103} by
\begin{eqnarray}
(\p_1\varphi)^2+\frac{(\p_2\varphi)^2}{n(0)^2} =B(y) & \text{on}\ \
\Sigma_0,\label{newentry}
\end{eqnarray}
with $B(y)$ a positive function not identical to $b_0^2$ and
satisfying either $B(y)\ge b_0^2$ or $B(y)\le b_0^2,$ then there
will be no any $C^2$ solution to the problem \eqref{0101},
\eqref{0102}, \eqref{newentry} and \eqref{0104}.
\end{theorem}

We remark that there is no any extra requirements such as the flow
to be subsonic in $\mathcal{M}$ to show the uniqueness. The only
assumption \eqref{acc} means that the flow is accelerating at the
exit. %However,
From the study of transonic flows \cite{K1}, and the non-uniqueness
results claimed by many authors based on numerical simulations (c.f.
\cite{So}), it seems that this assumption is necessary. Assertion
$(\rmnum{2})$ shows particularly that the subsonic--sonic flow is
not stable under perturbations of the density (or pressure, or
speed) of the gas at the entry, provided the flow is accelerating at
the exit, where it is sonic.       \\

The rest of this paper, Section 2, is devoted to the proof of
Theorem \ref{thm1}. The key point is to prove a version of Hopf
boundary point lemma applicable to points on characteristic
degenerate boundaries for degenerate elliptic equations, which is
Theorem \ref{thm2} stated in Section 2.1. Although there are many
impressive progresses in the study of degenerate elliptic equations
and mixed type equations in gas dynamics and other fields in these
years (see \cite{CDSW,CF,HH,Kim1,Kim2,K1,Mo1,OR,Z1} and references
therein), this generalized Hopf lemma seems to be new. It captures
the remarkable property that for subsonic--sonic flow or transonic
flow, the potential flow equation behaves like the heat equation
$-\p_1\fei+\p_{22}\fei=0$ near the sonic line (the line where the
equation is degenerate), and hence the lower order term $-\p_1\fei$
is essential in studying these degenerate elliptic or mixed type
equations (c.f. \cite{CF,YH}). This observation, obtained by
studying special subsonic--sonic flows and transonic flows in an
approximate nozzle, would be important for studying the flows in a
physical nozzle.

\section{Proof of Main Theorems}

\subsection{A generalized Hopf Lemma}
We first present a version of Hopf lemma
applicable to degenerate elliptic operators and characteristic
boundary points, which may be regarded as a generalization of
results like Lemma 3.4 in \cite{GT} or  Lemma 7.1.7 in \cite{Ta} and
is of independent interests.

Let $D$ be a bounded domain of class $C^2$ in $\mathbb{R}^n$ with
boundary $\p D$, and let
\begin{eqnarray}
L:=  \sum_{i,j =1}^n a^{ij}(x)\p_{ij} +  \sum_{i=1}^n b^i(x) \p_i +
c(x)
\end{eqnarray}
be a second order degenerate elliptic operator in $D$ (that is,
$a^{ij}(x)\xi_i\xi_j\ge0$ for all $\xi=(\xi_1, \cdots, \xi_n)\in
\mathbb{R}^n$ and $x\in D$). In what follows we assume that
$a^{ij}(x)$, $b^i(x)$, $c(x)$ are all bounded and continuous
functions in $D$. Let $P$ be a point on $\p D$ and assume, in a
neighborhood $U\subset\mathbb{R}^n$ of $P$, the boundary $\p D$ has
the form $\{x=(x_1, \cdots, x_n)\in\mathbb{R}^n: x_n - f(x_1,
\cdots, x_{n-1})=0\}$ with $f$ a $C^2$ function, and $U\cap D$ lies
in $\{x_n>f(x_1,\cdots,x_n)\}$. Then $\nu = (-\frac{\p f}{\p
x_1},\cdots, -\frac{\p f}{\p x_{n-1}},\, 1)$ is an interior normal
vector to $\p D$.  We have the following result.

\begin{theorem}\label{thm2}
Let $u$ be a $C^2(D)\cap C^1(\bar{D})$ function and attains a strict
minimum ${\mathrm{(}}$maximum${\mathrm{)}}$ at $P\in \p D$. Also
assume that, in a neighborhood $U\cap \bar{D}$ of $P$ we have $c=0$
and
\begin{eqnarray}\label{hopf}
 - \sum_{i=1}^{n-1} b^i \frac{\p f}{\p x_i}+ b^n - \sum_{i,j
=1}^{n-1} a^{ij}\frac{\p^2 f}{\p x_i \p x_j} > 0 \qquad \text{at}\ \
P.
\end{eqnarray}
Then if there holds
\begin{eqnarray}
Lu\le0 \ \ (\ge0)\qquad \text{in}\ \ D,
\end{eqnarray}
we have
\begin{eqnarray}
\frac{\p u}{\p \nu} > 0 \ \ (<0)\qquad \text{at}\ \  P.\label{a4}
\end{eqnarray}
If $c\le0$, the same conclusion holds provided that $u(P)\le0\,
(\ge0)$.
\end{theorem}

\begin{proof} First note that we may choose the following transformation
\begin{eqnarray}
y_1 = x_1 - x_1(P),~ y_2 = x_2 - x_2(P),~ \cdots , ~y_n = x_n -
f(x_1, \cdots, x_{n-1})
\end{eqnarray}
to map $P$ to $O$ in $y$-coordinates and straighten $\p D$ locally
to $y_n=0$ near the origin $O$. Denote the image of $U$ to be $V$.

It follows that the operator $L$ can be written in the form
\begin{eqnarray}
L &=& \sum_{i,j =1}^n a^{ij}(x)\frac{\p^2}{\p x_i \p x_j} +
\sum_{i=1}^n b^i(x) \frac{\p}{\p x_i} + c \nonumber\\
  &=& \sum_{k, l =1}^n
  \left(\sum_{i,j =1}^n a^{ij}\frac{\p y_k}{\p x_i}\frac{\p y_l}
  {\p x_j}\right)\frac{\p^2}{\p y_k \p
  y_l}\nonumber\\
  &&
  + \sum_{l =1}^n \left(\sum_{i =1}^n b^i \frac{\p y_l}{\p x_i} +
\sum_{i, j =1}^n a^{ij}\frac{\p^2 y_l}{\p x_i \p
x_j}\right)\frac{\p}{\p
y_l} + c\nonumber\\
  & =& \sum_{k,l =1}^n \alpha^{kl}(y)\frac{\p^2}{\p y_k \p y_l} +  \sum_{l = 1}^n
\beta^l(y) \frac{\p}{\p y_l}+ c,
\end{eqnarray}
where $\alpha^{kl}(y): = \sum_{i,j =1}^n a^{ij}\frac{\p y_k}{\p
x_i}\frac{\p y_l}{\p x_j}$ and $ \beta^l  := \sum_{i =1}^n b^i
\frac{\p y_l}{\p x_i} + \sum_{i, j =1}^n a^{ij}\frac{\p^2 y_l}{\p
x_i \p x_j}$.

Hence by the assumption \eqref{hopf} we have
\begin{eqnarray}
\beta^n (O) &=& \left(\sum_{i =1}^n b^i \frac{\p y_n}{\p x_i} +
\sum_{i, j =1}^n a^{ij}\frac{\p^2 y_n}{\p x_i \p
x_j}\right)(P)\nonumber\\
&=& \left(- \sum_{i =1}^{n-1} b^i \frac{\p f}{\p x_i} + b^n -
\sum_{i,j =1}^{n-1} a^{ij} \frac{\p^2 f}{\p x_i \p
x_j}\right)(P)\label{0110}\nonumber \\
&>& 0.
\end{eqnarray}
Now we introduce an auxiliary function $h$ of the form
\begin{eqnarray}
h(y) = - \sum_{i =1}^n y_i^2  + \mu y_n,
\end{eqnarray}
where $\mu$ is a positive constant yet to be determined. A direct
calculation gives
\begin{eqnarray}
L h (O)= - 2\sum_{i =1}^n \alpha^{ii}(O) + \mu \beta^n (O)
\end{eqnarray}
and
\begin{eqnarray}
\frac{\p h}{\p y_n}(O) =  \mu > 0.
\end{eqnarray}
Note that we have $\beta^n(O) > 0$. Hence $\mu$ may be chosen large
enough so that $Lh(O)>0$.

Let us pick up a small rectangle $D_1: = (-d_1, d_1) \times (-d_2,
d_2)\times ...\times (0,c) \subset V$, with $O \in \p D_1$ and $Lh
> 0$  through out the domain $D_1$. Corresponding to the function
$h$ there is an open set $ B = \{y \in \mathbb{R}^n: h(y) > 0\}$.
One sees from the construction of $h$ that $ K := D_1 \cap B$ is
non-empty, and the boundary of $K$ can be divided into two parts:
$K_1:= \p K \cap \p D_1$, $K_2:= \p K \cap \p B$, with $O \in K_2$.
Now since $ u-u(O)
> 0 $ on $K_1,$ there is a constant $\epsilon > 0$ for which the function
$w:= (u-u(O)) - \epsilon h \geq 0$ on $K_1$. This inequality is also
satisfied on $K_2$ where $h = 0$. Thus we have $Lw = Lu -cu(O)-
\epsilon L h \leq - \epsilon L h < 0$ in $K$, and $ w \geq 0$ on $\p
K$. The weak minimum principle (see the first paragraph of the proof
of Theorem 3.1 in \cite{GT}; it is here we need the matrix
$(a^{ij})$, hence $(\alpha^{kl})$, to be semi-positive definite) now
implies that $ w \geq 0$ in $K$. Taking the interior normal
derivative at the point $O$, where $w(O)=0$, we have $ \frac{\p
w}{\p y_n}(O) \geq 0, $ which implies
\begin{eqnarray}
 \frac{\p u}{\p y_n}(O) \geq \epsilon \frac{\p h}{\p
y_n}(O) = \epsilon \mu > 0.
\end{eqnarray}
Therefore we get, in the original $x$-coordinates,
\begin{eqnarray}
\frac{\p u}{\p\nu}(P)&=&\left(1+\sum_{i=1}^{n-1}\left(\frac{\p f}{\p
x_i}\right)^2\right)\frac{\p u}{\p y_n}(O)-\sum_{i=1}^{n-1}\frac{\p
u}{\p y_i}(O)\frac{\p f}{\p x_i}(P)\nonumber\\
&>&0\label{a14}
\end{eqnarray}
as required, since  there should hold $\frac{\p u}{\p y_i}(O)=0$ for
$i=1,\cdots, n-1$ due to the fact that $u$ attains a minimum at $O$.

The case that $u$ attains a strict maximum at $P\in\p D$ can be
proved by applying the above proved result to $-u$. For $c\le 0,$
the proof is similar.
\end{proof}

\begin{remark}
As pointed out in \cite{GT} (p. 34), \eqref{a4} can be replaced by
\begin{eqnarray}
\liminf_{x\rightarrow P}\frac{|u(x)-u(P)|}{|x-P|}>0\ (<0),
\end{eqnarray}
where the angle between the vector $x-P$ and the inner normal vector
of $\p D$ at $P$ is less than $\pi/2-\delta$ for some fixed positive
number $\delta.$ We require $u\in C^1(\bar{D})$ is in essence just
to obtain that the tangential  derivatives of $u$ along $\p D$   are
zero at $P,$ which is used to derive \eqref{a14}.
\end{remark}

\subsection{Proof of Theorem \ref{thm1}}

We first prove part $(\rmnum{1})$ of Theorem \ref{thm1}.

Suppose $\fei$ is a $C^2$ solution to problem
\eqref{0101}--\eqref{0104}. Then by \eqref{0101}, it is
straightforward to check that $\psi=\fei_b-\fei$ satisfies the
following equation
\begin{eqnarray}\label{le}
\sum_{i,j=1}^2 a^{ij}(\fei_b)\p_{ij}\psi+ b^1\p_1\psi+b^2\p_2\psi=0,
\end{eqnarray}
where
\begin{eqnarray}
&&a^{11}(\fei_b)=n(x)^2(c_b^2-(\p_1\fei_b)^2), \\
&&a^{12}(\fei_b)=a^{21}(\fei_b)=-\p_1\fei_b\p_2\fei_b=0,\\
&&a^{22}(\fei_b)=c_b^2-\frac{1}{n(x)^2}(\p_2\fei_b)^2=c_b^2,
\end{eqnarray}
with $c_b=\sqrt{c_0-\frac{\gamma-1}{2}(\p_1 \fei_b)^2}$ the sonic
speed corresponding to $\fei_b$, and
\begin{eqnarray}
&& b^1=-\left(\frac{\gamma+1}{2}n(x)^2\p_{11}\fei+\frac{\gamma-1}{2}
\p_{22}\fei\right)\p_1(\fei+\fei_b)\nonumber\\
&&\qquad
+n(x)n'(x)\left(c_b^2+\frac{\gamma-1}{2}\p_1\fei\p_1(\fei+\fei_b)\right),\\
&&b^2=\frac{\gamma-1}{2}\p_2\fei\p_{11}\fei+2\p_1\fei\p_{12}\fei+
\frac{\gamma+1}{2}\frac{1}{n(x)^2}\p_2\fei\p_{22}\fei\nonumber\\
&&\qquad+\frac{\gamma-3}{2}\frac{n'(x)}{n(x)} \p_1\fei\p_2\fei.
\end{eqnarray}
Recalling that $\fei_b$ is a subsonic--sonic flow, so equation
\eqref{le} is a linear degenerate elliptic equation of $\psi$: It is
strictly elliptic in $\mathcal{M}\cup\Sigma_0$ and degenerate on
$\Sigma_1$.

The boundary conditions of $\psi$ are
\begin{eqnarray}
&l_0\cdot D\psi=0 & \text{on}\ \ \Sigma_0,\label{lc0}\\
&l_1\cdot D\psi=0 & \text{on}\ \ \Sigma_1,\label{lc1}
\end{eqnarray}
where $D\psi=(\p_1\psi,\  \p_2\psi)$ and
$$l_0=(\p_1(\fei+\fei_b),\ \p_2\fei/n(0)^2),\ \
l_1=(\p_1(\fei+\fei_b),\ \p_2\fei).$$ By $(H_2)$ and the fact
$\p_1\fei_b>0$, \eqref{lc0} and \eqref{lc1} are both
oblique derivative conditions.\\

Now we apply maximum principles to the problem \eqref{le},
\eqref{lc0} and \eqref{lc1}. If $\psi$ is not a constant, by strong
maximum principle which is valid in any compact subset of
$\mathcal{M}$, the maximum and minimum of $\psi$ can only be
achieved on $\Sigma_0$ or $\Sigma_1$. By the classical Hopf boundary
point lemma (Lemma 3.4 in \cite{GT}), it is only possible to achieve
the extremum on $\Sigma_1$, i.e., the boundary where the equation is
degenerate.

To show $\psi$ is a constant, by Theorem \ref{thm2} and an argument
of contradiction, we just need to show the validity of \eqref{hopf}.
In our situation, it is exactly
\begin{eqnarray}\label{224}
b^1<0\quad \text{on}\ \Sigma_1.
\end{eqnarray}
The calculation is straightforward and we omit it (c.f. \eqref{0110}
and also note that $y_n=-(x-1)$ now).

Since $\p_1\fei_b>0,\ \p_1\fei\ge0$ and $n'(1)=0, n(1)=1$ by
$(H_1)(H_2)$, \eqref{224} is guaranteed by
\begin{eqnarray}\label{key}
\frac{\gamma+1}{2}\p_{11}\fei+\frac{\gamma-1}{2} \p_{22}\fei>0 \quad
\text{on}\ \Sigma_1.
\end{eqnarray}
We show this is true under the assumption \eqref{acc}.

In fact, by the equation \eqref{0101}, we obtain that on $\Sigma_1,$
\begin{eqnarray}\label{225}
c^2\Delta\fei=(\p_1\varphi)^2\p_{11}\varphi+2\p_1\varphi\p_2\varphi\p_{12}\varphi
+(\p_2\varphi)^2\p_{22}\varphi.
\end{eqnarray}
Differentiating the boundary condition \eqref{0104} with respect to
$y$ and multiplying $\p_2\fei$, it follows
\begin{eqnarray}
\p_1\fei\p_2\fei\p_{12}\fei+(\p_2\fei)^2\p_{22}\fei=0.
\end{eqnarray}
Substituting this in \eqref{225} and note that
$c^2=(\p_1\fei)^2+(\p_2\fei)^2,$  then
\begin{eqnarray}
(c^2+(\p_2\fei)^2)\Delta\fei=c^2\p_{11}\fei>0.
\end{eqnarray}
So $\Delta\fei>0$ on $\Sigma_1$ and \eqref{key} follows. (Note that
$c=b_1>0$ on $\Sigma_1.$) This finishes the proof of part
$(\rmnum{1})$ of Theorem \ref{thm1}. \\

We continue to prove part $(\rmnum{2})$ of Theorem \ref{thm1}. The
only difference is the boundary condition \eqref{lc0}. Now it should
be written as
\begin{eqnarray}
l_0\cdot D\psi=b_0^2-B(y) & \text{on}\ \ \Sigma_0.\label{lc001}
\end{eqnarray}
We consider the case $B\ge b_0^2.$ Suppose $\psi$ is not a constant.
Then by the classical Hopf lemma, the minimum of $\psi$ can only be
achieved on $\Sigma_1.$ But this is contradictory to \eqref{224} and
Theorem \ref{thm2}. So $\psi$ must be a constant. Hence on
$\Sigma_0$ there should hold $b_0^2\equiv B(y).$ However, we have
assumed that $B$ is not identical to $b_0^2$, so the only
possibility is that there is no any $C^2$ solution of the problem
\eqref{0101}, \eqref{0102}, \eqref{newentry} and \eqref{0104}. The
case $B\le b_0^2$ can be proved similarly by considering the maximum
of $\psi$.

This finished the proof of Theorem \ref{thm1}.

\begin{remark}
We see the equation \eqref{le} is reduced to
\begin{eqnarray}
c_b^2\p_{22}\psi+b^1\p_1\psi+b^2\p_2\psi=0
\end{eqnarray}
on $\Sigma_1.$ So \eqref{224} implies that this equation is similar
to a heat equation. This type of degeneracy also occurs, for
example, in shock reflection phenomena (c.f. \cite{CF}).
\end{remark}

\bigskip
{\bf Acknowledgments.} P. Liu is supported in part by the National
Science Foundation of China under Grant No.\,10601017 and 10871126.
H. Yuan is supported in part by NNSF of China under grant
(10901052), Shanghai Chenguang Program (09CG20), a Special Research
Fund for Selecting Excellent Young Teachers of the Universities in
Shanghai sponsored by Shanghai Municipal Education Commission, and
the National Science Foundation (USA) under Grant DMS-0720925. The
authors thank sincerely the referees for providing many valuable
comments and suggestions!

%%%----------------------------------------------------------------------------


\begin{thebibliography}{99}

\bibitem{CDSW}
G.-Q. Chen, C.~M. Dafermos, M. Slemrod and  D. Wang, {\it On
two-dimensional sonic-subsonic flow,} Comm. Math. Phys. (3)271
(2007), 635--647.


\bibitem{CF}
G.-Q. Chen and M. Feldman, {\it Global solutions to shock reflection
by large--angle wedges for potential flow,} Ann. of Math. (2006), in
press.


\bibitem{CY}
G.-Q. Chen and H. Yuan, {\it Uniqueness of transonic shock solutions
in a duct for steady potential flow,}  J. Differential Equations
(2)247 (2009), 564--573.


\bibitem{CoF}
R. Courant and K.~O. Friedrichs, {\it Supersonic Flow and Shock
Waves}, Interscience Publishers, Inc., New York, 1948.


\bibitem{GT}
D. Gilbarg and N.~S. Trudinger, {\it Elliptic Partial Differential
Equations of Second Order,} second edition, Springer, Berlin--New
York, 1983.


\bibitem{HH}
Q. Han and J.-X. Hong, {\it Isometric Embedding of Riemannian
Manifolds in Euclidean Spaces}, Mathematical Surveys and Monographs
130, American Mathematical Society, Providence, RI, 2006.


\bibitem{Kim1}
E.~H. Kim, {\it Subsonic solutions for compressible transonic
potential flows,} J. Differential Equations 233 (2007), 276--290.


\bibitem{Kim2}
E.~H. Kim, {\it Subsonic solutions to compressible transonic
potential problems for isothermal self-similar flows and steady
flows,} J. Differential Equations 233 (2007), 601--621.


\bibitem{K1}
A.~G. Kuz'min, {\it Boundary--Value Problems for Transonic Flow},
John Wiley {\&} Sons, West Sussex, 2002.


\bibitem{LY1}
L. Liu and H. Yuan,  {\it Stability of cylindrical transonic shocks
for two--dimensional steady compressible Euler cystem,}  J.
Hyperbolic Diff. Equ. (2)5 (2008), 347--379.


\bibitem{LY}
L. Liu and H. Yuan, {\it Global uniqueness of transonic shocks in
divergent nozzles for steady potential flows,} SIAM J. Math. Anal.
(2009), acceptd for publication.


\bibitem{Mo1}
C.~S. Morawetz, {\it Mixed equations and transonic flow}, J.
Hyperbolic Diff. Equ. (1){1} (2004), 1--25.


\bibitem{OR}
{$\mathrm{O.~A.\ Ole\check{i}nik}$} and  {$\mathrm{E.~V.\
Radkevi\check{c}}$}, {\it Second Order Equations with Nonnegative
Characteristic Form,} Translated from the Russian by Paul C. Fife,
Plenum Press, New York--London, 1973.


\bibitem{SS}
L.~M. Sibner and R.~J. Sibner, {\it Transonic flows on axially
symmetric torus}, J. Math. Anal. Appl. {72} (1979), 362--382.


\bibitem{So}
P. $\mathrm{\check{S}ol\acute{i}n}$ and K. Segeth, {\it
Non-uniqueness of almost unidirectional inviscid compressible flow},
Appl. Math. (3)49 (2004), 247--268.


\bibitem{Ta}
K. Taria, {\it Diffusion Process and Partial Differential
Equations,} Elsivier, Singapore, 2004.


\bibitem{Yu1}
{H. Yuan}, {\it Examples of steady subsonic flows in a
convergent--divergent approximate nozzle,} J. Differential Equations
(7)244 (2008), 1675--1691.


\bibitem{YH}
H. Yuan and Y. He, {\it Transonic potential flows in a
convergent--divergent approximate nozzle,} J. Math. Anal. Appl.
(2)353 (2009), 614--626.


\bibitem{Z1}
Y. Zheng, {\it Existence of solutions to the transonic
pressure--gradient equations of the compressible Euler equations in
elliptic regions}, Comm. Partial Differential Equations (11--12)22
(1997), 1849--1868.


\end{thebibliography}
\end{document}